\documentclass[12pt]{amsart}
\usepackage[mathscr]{eucal}
\usepackage{amssymb}
\usepackage{latexsym}
\usepackage{amsthm}

\theoremstyle{plain}

\newtheorem*{theo}{Theorem}
\newtheorem{prop}{Proposition}
\newtheorem{lem}{Lemma}
\newtheorem{cor}{Corollary}

\title[Invariant subspaces with no generator]{Invariant subspaces with no
generator \\ and a problem of H. Helson}
\author{ Jun-ichi Tanaka}

\dedicatory{Dedicated to the memory of Henry Helson}

\address{Department of Mathematics, School of Education, Waseda
University, Shinjuku, Tokyo 169-8050, Japan}

\email{jtanaka@waseda.jp}

\date{Received by editors November 21, 2009, and in revised form, January 17, 2013}


\thanks{Partially supported by NSF grant no.\,0649765}

\newcommand{\fO}{(\Omega, \{U_t\}_{t\in  \mathbb{R}})}
\newcommand{\fK}{(K, \{T_t\}_{t\in  \mathbb{R}})}

\newcommand{\fKT}{(K\times \mathbb{T}, \{S_t\}_{t\in  \mathbb{R}})}

\begin{document}
\maketitle

\begin{abstract}
In the almost-periodic context, the $H_0^2-$space cannot be
generated by one of its elements. Together with a cocycle argument,
this implies that there exist all kinds of invariant subspaces without
single generator, from which we answer some questions on
invariant subspace theory.

\end{abstract}

\bigskip
\section{Introduction}

The theory of invariant subspaces has been developed in the context
of compact abelian groups with ordered duals, which is a natural
generalization of such a theory on the unit circle $ \mathbb{T}$. Many
classical results extend to these cases, nevertheless, one also
meets new difficulties. The purpose of this paper is to resolve a
longstanding problem formulated by H. Helson in the 1950s.

Let $\Gamma$ be a countable dense subgroup of the real line
$ \mathbb{R}$, endowed with the discrete topology. Then the dual
group $K$ of $\Gamma$ is a compact abelian group that is metrizable.
For $\lambda$ in $\Gamma$, it is customary to denote by $\chi_\lambda$ the character
on $K$ defined by $\chi_\lambda(x)\,=\,x(\lambda)$. Let $\sigma$ be
the normalized Haar measure on $K$.  A function $\phi$ in
$L^1(\sigma)$ is {\itshape analytic} if its Fourier coefficients
\begin{equation}
a_\lambda(\phi)\;=\;\int_K\,\phi\,\overline{\chi_\lambda}\,d\sigma
\label{(1)}
\end{equation}
vanish for all negative $\lambda$ in $\Gamma$. The {\itshape Hardy
space} $H^p(\sigma), 1\le p \le \infty$, is defined to be the space
of all analytic functions in $L^p(\sigma)$. For technical reasons,
it is useful to define $H^p_0(\sigma)$ as the subspace of all
$\phi$ in $H^p(\sigma)$ with $a_0(\phi)\,=\,0$. A (weak*-, if
$p=\infty$) closed subspace $\mathfrak{M}$ of $L^p(\sigma)$ is
{\itshape invariant \/} if $\,\mathfrak{M}$ contains $\chi_\lambda
\mathfrak{M}$ for all positive $\lambda$ in $\Gamma$. When the
inclusion is strict, $\,\mathfrak{M}$ is said to be {\itshape simply
\/} invariant. Of course, both $H^p(\sigma)$ and $H^p_0(\sigma)$ are
simply invariant subspaces of $L^p(\sigma)$. If $\phi$ is in
$L^p(\sigma)$, and let $\mathfrak{M}[\phi]$ denotes the smallest
invariant subspace of $L^p(\sigma)$ containing $\phi$, then
$\phi$ is called a {\itshape single generator} of $\mathfrak{M}[\phi]$.
Recall that a function of modulus one is said to be {\itshape unitary}
and an analytic unitary function is called an {\itshape inner} function.
We say a function $\phi$ in $H^p(\sigma)$ is
{\itshape outer\/} if it satisfies that
$$
\log \mid a_0(\phi) \mid \;=\; \int_K \log \mid \phi \mid d\sigma
\;>\;-\infty \, .
$$
Let $1\le q \le p \le \infty$, and let $\mathfrak{M}$ be a simply
invariant subspace of $L^p(\sigma)$. It follows from the properties
of outer functions that $[\mathfrak{M}\cap L^\infty(\sigma)]_q\cap
L^p(\sigma)\,=\,\mathfrak{M}$, where $[\mathfrak{M}\cap
L^\infty(\sigma)]_q$ is the closure of $\mathfrak{M}\cap
L^\infty(\sigma)$ in $L^q(\sigma)$ (see \cite[Chapter V, Section
6]{G2} for details).
This fact assures that there is a one-to-one correspondence
between the invariant subspaces in $L^p(\sigma)$ and those in
$L^q(\sigma)$. Therefore, in dealing with invariant subspaces,
we may restrict our attention to the case of $p=2$, in
which Hilbert space theory works well.
It follows from Szeg\"o's theorem that $\phi$ is a single
generator of $H^2(\sigma)$ if and only if $\phi$ is outer in
$H^2(\sigma)$.
However, it has been unknown for a long time whether every simply
invariant subspace is singly generated or not. In the
literature this has come to be known as the \textit{single generator
problem} (refer to \cite[\S 5.4]{H1},
\cite[Remark, p.\,158]{G1} and \cite[p.\,138 and p.\,177]{G2}).
The difficulty seems
to center on the case of invariant subspace $H^2_0(\sigma)$.
In \cite[p.\,183]{HL}, it is raised in an equivalent form
in connection with stochastic processes.

Our objective in this note is to show a negative answer to this
problem in the almost periodic settings:

\begin{theo}
\label{theo}
 \; The invariant subspace $H^2_0(\sigma)$ cannot be generated by one of its elements.
\end{theo}

To the best of author's knowledge, $H^2_0(\sigma)$ is the first known
example of invariant subspace which cannot be singly generated.
On the other hand, by \cite[\S 5.3, Theorem 33]{H1}, it was shown that
every invariant subspace is generated by
two of its elements. In more general setting, we can artificially make
$H_0^2$-spaces to have a single generator.

For each $t$ in $ \mathbb{R}$, let us denote by $e_t$ the element of $K$
defined by $e_t(\lambda)\,=\,e^{i\lambda t}$ for $\lambda$ in
$\Gamma$. The map sending $t$ to $e_t$ embeds $ \mathbb{R}$
continuously onto a dense subgroup of $K$. Define a one-parameter
group $\{T_t\}_{t\in  \mathbb{R}}$ of homeomorphisms on $K$ by
\begin{equation}
T_t\,x\;=\;x + e_t\,, \qquad x \in K.
\label{(2)}
\end{equation}
Then the pair $\fK$ is a strictly ergodic flow, for which
$\sigma$ is the unique invariant probability measure. The flow $\fK$
is called an \textit{almost periodic flow}, because if $\phi$ is
continuous on $K$, then $t\rightarrow \phi(x+e_t)$ is a uniformly
almost periodic function with exponents in $\Gamma$.
Let $H^\infty(dt/\pi(1+t^2))$ be the space of all boundary functions
of bounded analytic functions in the upper half-plane $ \mathcal{H}$,
and let $H^p(dt/\pi(1+t^2)), 1\le p < \infty,$ be the
closure of $H^\infty(dt/\pi(1+t^2))$ in $L^p(dt/\pi(1+t^2))$.
For a function $u(x,t)$ on $K\times \mathbb{ R}$, the assertion
``\,$t \rightarrow u(x,t)$ for $\sigma-a.e. \,x$ in $K$\,''
is sometimes abbreviated to ``\textit{almost every} $t \rightarrow u(x,t)$ ''.
Then $\phi$ in $L^p(\sigma)$ lies in $H^p(\sigma)$
if and only if almost every $t \rightarrow \phi(x+e_t)$ lies in
$H^p(dt/\pi(1+t^2))$.
This fact enables us to define Hardy spaces
on every ergodic flow (see the end of the next section).

Let $\mathfrak{M}$ be a simply invariant subspace of $L^2(\sigma)$.
Set $\mathfrak{M}_\lambda\,=\,\chi_\lambda \mathfrak{M}$
for each $\lambda$ in $\Gamma$. Define
$$
\mathfrak{M}_+\;=\;\bigwedge_{\lambda<0}\;\mathfrak{M}_\lambda
\qquad \text{and} \qquad \mathfrak{M}_-\;=\;
\bigvee_{\lambda>0}\;\mathfrak{M}_\lambda.
$$
Since these spaces are at most one dimension apart,  $\mathfrak{M}$
coincides with either or both its versions $\mathfrak{M}_+$ and
$\mathfrak{M}_-$. When $\/ \mathfrak{M}\,=\,\mathfrak{M}_+$,
$\mathfrak{M}$ is said to be {\itshape normalized}.  For $\phi$ in
$L^2(\sigma)$, the subspace $\mathfrak{M}[\phi]$ is simply invariant
if and only if
\begin{equation}
\int_{-\infty}^\infty \, \log \vert \phi(x+e_t)\vert \,
\frac{dt}{1+t^2}\; > \; -\infty\,,\qquad  \sigma-a.e.\; x \in K,
\label{(3)}
\end{equation}
(see \cite[\S 3.3, Theorem 22 ]{H1}). It is well-known that there is a
function $\phi$ in $L^2(\sigma)$ satisfying the inequality \eqref{(3)},
while $\log\vert\phi\vert$ does not belong to $L^1(\sigma)$. Our
Theorem asserts that any such function $\phi$ must satisfy
$\mathfrak{M}[\phi]_+\,=\,\mathfrak{M}[\phi]_-$.

A unitary Borel function $A(x,t)$ on $K \times \mathbb{ R}$ is said
to be a {\itshape cocycle} on $K$ if $A(x,t)$ satisfies the
\textit{cocycle identity}
$$
A(x,t+s) \; = \; A(x,t)\cdot A(x + e_t,s)\,, \qquad (x,s,t) \in
K\times  \mathbb{R} \times  \mathbb{R}\,.
$$
We identify two cocycles which differ only on a set of $d\sigma
\times dt-$measure zero in $K\times  \mathbb{R}$. A one-to-one
correspondence is established between normalized invariant subspaces
and cocycles (as discussed in \cite[\S 2.3]{H1}). More precisely,
let $\/\mathfrak{M}$ be a simply invariant subspace
of $L^2(\sigma)$ with cocycle $A(x,t)$. Then a function
$\phi$ in $L^2(\sigma)$
lies in $\mathfrak{M}_+$ if and
only if almost every $t \rightarrow A(x,t)\phi(x+e_t)$ lies in
$H^2(dt/\pi(1+t^2))$ (see
\cite[\S 3.2]{H1}). It is easy to see that
$\mathfrak{M}_+\,\not=\,\mathfrak{M}_-$ if and only if
$\mathfrak{M}_+\,=\, qH^2(\sigma)$ for some unitary function $q$
on $K.$ Then the cocycle of $\/\mathfrak{M}$ has
the form $\,q(x) \cdot \overline{q(x+e_t)},$ which is called
a {\itshape coboundary}. If a cocycle is a coboundary
multiplied by $\exp (i \alpha t)$ for some
$\alpha$ in $ \mathbb{R}$, then such a cocycle is said to be
\textit{trivial}. A trivial cocycle $\exp (i \alpha t)$ is not a
coboundary only if $\alpha$ lies in $ \mathbb{R}\setminus \Gamma$.

We already know from \cite{H3} and \cite{MS} that some singly generated
subspaces have nontrivial cocycles, but we can strengthen this fact
by noting the following:

\begin{cor}
\label{cor1}
Let $\,\mathfrak{M}$ be a simply invariant subspace of $L^2(\sigma)$.
If the cocycle of $\/\mathfrak{M}$ is trivial, then
$\mathfrak{M}_-$ has no single generator.
In other words, if $\,\mathfrak{M}_-$ is singly generated, then
the cocycle of $\,\mathfrak{M}$ is always nontrivial,
so that $\; \mathfrak{M}_+ \,=\,\mathfrak{M}_-$.
\end{cor}

A cocycle with values in $\{-1, 1\}$ is called a
\textit{real} cocycle. It follows from \cite{HP} that there
exist real cocycles which are nontrivial.

\begin{cor}
\label{cor2}
Let $\,\mathfrak{M}$ be a simply invariant subspace of $L^2(\sigma)$
with real cocycle. Then $\mathfrak{M}_-$ has no single generator.
\end{cor}

A cocycle $A(x,t)$ is said to be \textit{analytic} if almost every
$t \rightarrow A(x,t)$ lies in $H^\infty(dt/\pi(1+t^2))$. Then
a normalized invariant subspace with analytic cocycle contains
always $H^2(\sigma)$. We say that an analytic cocycle $A(x,t)$ is
a \textit{Blaschke} or a \textit{singular} cocycle, if almost every
$t\rightarrow A(x,t)$ is an inner function of that type in
$H^\infty(dt/\pi(1+t^2))$. Two cocycles are called
\textit{cohomologous} if one is a coboundary times the
other. It is known that every cocycle is cohomologous to a
Blaschke cocycle in some restricted class
(see \cite[\S 4.6, Theorem 26]{H1} and \cite{T1}). This fact makes
Blaschke cocycles so important for the subject.
Using our Theorem, we may answer some questions on analytic cocycles:

\begin{cor}
\label{cor3}
In the class of analytic cocycles, the following properties hold:
\begin{enumerate}
\renewcommand{\labelenumi}{(\alph{enumi})}
\item There is a Blaschke cocycle not being cohomologous to any
singular cocycle.
\item There is a Blaschke cocycle not having exactly the same
zeros as any function in $H^2(\sigma)$.
\end{enumerate}
\end{cor}

It would be helpful to understand the basic idea behind the
proof of our Theorem. On the one hand, we claim that if $\phi$
is a single generator of $H^2_0(\sigma)$, then $\phi$ must
have a very special form. Assume
that $\Gamma$ is the smallest group determined by the nonzero
Fourier coefficients of $\phi$ (see below for details).
Similarly, let $\Lambda$ be the smallest group determined by the
nonzero coefficients of $\vert \phi \vert$. Since $\Lambda$ is a subgroup of
$\Gamma$, the dual group of $\Lambda$ is represented as $K/H$,
where $H$ is the annihilator of $\Lambda$ in $K$.
Let $\tau$ be the normalized Haar measure on $K/H$, and fix
an element  $\alpha$ in $\Gamma$ with $a_\alpha(\phi)\not=0$.
Then it can be shown that
$\overline{\chi}_\alpha\phi$ lies in
$L^2(\tau)$ and generates the simply invariant subspace of
$L^2(\tau)$ with trivial cocycle $\exp (i \alpha t)$. We also see that
$\alpha$ is independent of $\Lambda$, meaning that $n\alpha$
lies in $\Lambda$ only for $n=0$ in the integer group
$ \mathbb{Z}$.  This implies that $K$ and $d\sigma$ are  respectively
identified with $K/H\times \mathbb{T}$ and
$d\tau\times d\theta/2\pi$, since $H$ is regarded as  $ \mathbb{T}$.
Thus, for each single generator $\phi$ of  $H^2_0(\sigma)$, we derive
that $\Gamma\,\not=\,\Lambda$. On the other hand, if
$H^2_0(\sigma)$ is singly generated, we may construct
a generator $\phi$ of $H^2_0(\sigma)$ with the property that
$\Gamma=\Lambda$, which contradicts the existence
of single generator of $H^2_0(\sigma)$.

\medskip
In the next section, we establish some notation and elementary
facts about invariant subspaces in the almost periodic setting. Using
group characters, we develop certain properties of single generators
of $H^2_0$-spaces in Section 3. In Section 4, the proof of our Theorem
is provided and then Corollaries are proved by using  a lemma on cocycles.
We conclude the paper with some remarks in Section 5.

\medskip
We refer the reader to \cite{DG}, \cite[Chapter VII]{G2}, \cite{H1}
and \cite[Chapter VIII]{R} for further details on analyticity on
compact abelian groups. Basic results concerning the Hardy
space theory based on uniform algebras can be found in
\cite[Chapter IV]{G2} and \cite{M1}.

\medskip
Part of this work was done while the author was
visiting the University of North Carolina at Chapel Hill, and
he would like to acknowledge the hospitality of the Department
of Mathematics. Especially, he would like to express his sincere gratitude
to Professors Joe Cima and Karl Petersen for helpful discussions.
Thanks are due to the referee as well for his valuable suggestions,
which improved the first version of this paper so much.

\newpage
\section{Extension of almost periodic functions}\label{S2}

It is easy to show that a function $\phi$ in $H^2(\sigma)$ is outer
if and only if $a_0(\phi) \not= 0$ and almost every $t\rightarrow
\phi(x+e_t)$ is outer in $H^2(dt/\pi(1+t^2))$. A weak version of
this fact stated below is often used in what follows:

\begin{lem}
Let $\/\mathfrak{M}$ be a simply invariant subspace of $L^2(\sigma)$
with cocycle $A(x,t)$. A function $\phi$ in $L^2(\sigma)$ generates
$\,\mathfrak{M}_-$ if and only if \/$ \log\vert\phi\,\vert$ does not lie
in $L^1(\sigma)$ and almost every $t \rightarrow A(x,t)\phi(x+e_t)$
is outer in $H^2(dt/\pi(1+t^2))$. In particular,
$H^2_0(\sigma)$ is singly generated by $\phi$ if and only if
$a_0(\phi)\,=\,0$ and almost every $t \rightarrow \phi(x+e_t)$
is outer in $H^2(dt/\pi(1+t^2))$.

\label{lem1}
\end{lem}

\begin{proof}
\;\;  Suppose that $\mathfrak{M}[\phi]\,=\,\mathfrak{M}_-$ for
$\phi$ in $L^2(\sigma)$. If
$\log\vert\phi\vert$ lies in $L^1(\sigma)$, then there is a unitary
function $q$ on $K$ such that $\mathfrak{M}[\phi]\,=\,qH^2(\sigma)$
by Szeg\"o's theorem. This implies that $\mathfrak{M}[\phi]\,
\not=\,\mathfrak{M}_-$, so $\log\vert\phi\vert$ cannot lie in
$L^1(\sigma)$. Let $B(x,t)$ be the analytic cocycle defined by the
inner part of $t \rightarrow A(x,t)\phi(x+e_t)$. Let $\mathfrak{N}$
be the invariant subspace with cocycle $A\overline{B}(x,t)$. By
\cite[\S 3.2, Theorem 21]{H1}, we see that $\mathfrak{N}_-$
is contained in $\mathfrak{M}_-$. On the other
hand, since almost every $t \rightarrow
A\overline{B}(x,t)\psi(x+e_t)$ lies in $H^2(dt/\pi(1+t^2))$ for each
$\psi$ in $\mathfrak{M}[\phi]$, $\mathfrak{N}_+$ includes
$\mathfrak{M}[\phi]$. This shows that
$\mathfrak{N}_+\,=\,\mathfrak{M}_+$, so $B(x,t)\,\equiv 1$. Then
almost every $t \rightarrow A(x,t)\phi(x+e_t)$ is outer in
$H^2(dt/\pi(1+t^2))$.

Conversely, suppose that $\mathfrak{M}[\phi]$ is contained strictly
in $\mathfrak{M}_-$. Then there is a nonzero function $q$ in
$\mathfrak{M}_-$ such that
$$
\int_{K}\,\psi\phi\overline{q}\,d\sigma\;=\;0\,, \qquad \psi\in
H^\infty(\sigma)\,.
$$
This shows that $\phi\overline{q}$ lies in $H^1(\sigma)$, so almost
every $t \rightarrow \phi\overline{q}(x+e_t)$ lies in
$H^1(dt/\pi(1+t^2))$. Notice that $t \rightarrow A(x,t)q(x+e_t)$ is
in $H^2(dt/\pi(1+t^2))$. Since
$$
\phi(x+e_t)\overline{q(x+e_t)}\;=\;
A(x,t)\phi(x+e_t)\overline{A(x,t)}\overline{q(x+e_t)},
$$
and since $t \rightarrow A(x,t)\phi(x+e_t)$ is outer in
$H^2(dt/\pi(1+t^2))$, we see that almost every
$t \rightarrow \overline{A(x,t)q(x+e_t)}$ is also
in $H^2(dt/\pi(1+t^2))$.
This shows that
$t \rightarrow A(x,t)q(x+e_t)$
is constant for $\sigma-a.e. \,x$ in $K$, and so is $t \rightarrow
\vert q(x+e_t)\vert$. It follows from the ergodic theorem that
$\vert q(x)\vert$ is constant. We then assume $q$ is a unitary
function on $K$. Therefore, $A(x,t)$ is the coboundary
$q(x)\overline{q(x+e_t)}$ and $\mathfrak{M}_-\,=\,qH^2_0(\sigma)$.
Thus $q$ does not lie in $\mathfrak{M}_-$, which is a contradiction.

The last part of assertion
follows from the fact that the cocycle of
$H^2(\sigma)$ equals $1$. Under the assumption that almost
every $t \rightarrow \phi(x+e_t)$ is outer in $H^2(dt/\pi(1+t^2))$,
we see easily $a_0(\phi)\,=\,0$ if and only if
$ \log\vert\phi\,\vert$ does not lie in $L^1(\sigma)$.
Then $\mathfrak{M}[\phi]\,=\,H^2_0(\sigma)$, so the proof is
complete.
\end{proof}

Let $L^1(dt)$ be the usual Lebesgue space on $ \mathbb{R}$. Using
$\{T_t\}_{t\in  \mathbb{R}}$, one may convolve a function $\phi$ in
$L^p(\sigma), 1 \le p < \infty,$ with a function $f$ in $L^1(dt)$ by
setting
$$
(\phi\ast f)(x)\;=\;\int_{-\infty}^\infty \,\phi(x+e_t)f(-t)\,dt
\;=\;\int_{-\infty}^\infty \,\phi(x-e_t)f(t)\,dt\,,
$$
where the integral is a Bochner integral. When $p=\infty$,
the convolution $\phi\ast f$ is defined in the
same way as the weak*-convergent integral. Under the operation
of convolution, $L^p(\sigma)$ becomes an $L^1(dt)$-module such that
$$
\Vert \phi\ast f \Vert_p \;\le\;
\Vert \phi \Vert_p\Vert f \Vert_1\,,\qquad \phi \in L^p(\sigma)\,,
$$
for $f$ in $L^1(dt)$.
The Fourier transform $\hat{f}$ of $f$
is defined by the formula

\begin{equation}
\hat{f}(\lambda)\;=\;\int_{-\infty}^\infty f(t)e^{-i\lambda t}\,dt,
\qquad \lambda \in  \mathbb{R}\,,
 \label{(4)}
\end{equation}
as usual. We see easily
$a_\lambda(\phi\ast f)\,=\,a_\lambda(\phi)\hat{f}(\lambda)$,
if $\lambda$ is in $\Gamma$.
The Poisson kernel $P_{ir}(t)$ for $ \mathcal{H}$ is given
by $P_{ir}(t)\,=\,r/\pi(t^2+r^2)$, for an $r>0$. If $\phi$ is
in $L^1(\sigma)$, then the convolution $\phi\ast P_{ir}$ is
considered as the Poisson integral of $t \rightarrow \phi(x+e_t)$,
that is,
$$
(\phi \ast P_{ir})(x+e_s)\;=\; \int_{-\infty}^\infty
\,\phi(x+e_t)P_{ir}(s-t)\,dt\,.
$$

\begin{lem}
Suppose that $H^2_0(\sigma)$ is singly generated. Then we obtain the
following properties:

\begin{enumerate}
\renewcommand{\labelenumi}{(\alph{enumi})}
\item There is a single generator of $H^2_0(\sigma)$ that is
bounded.

\item If $\phi$ is a bounded generator of $H^2_0(\sigma)$,
then so is each of the functions $\phi \ast P_{ir}$ with $r>0$ and $\phi^n$
for $n=1,2,\cdots.$
\end{enumerate}
\label{lem2}
\end{lem}

\begin{proof}
\;\; Let $\psi$ be a single generator of $H^2_0(\sigma)$.
Then there is an outer function $h$ in $H^2(\sigma)$
such that $\vert h \vert\,=\, \min (1, \vert \psi \vert^{-1})$.
It follows from Lemma \ref{lem1} that the bounded function
$\psi h$ generates $H^2_0(\sigma)$, thus we obtain (a).

To show (b), we observe that $t\rightarrow (\phi \ast P_{ir})(x+e_t)$
as well as $t\rightarrow \phi^n(x+e_t)$ is outer in
$H^2(dt/\pi(1+t^2))$ for $\sigma-a.e. \,x$ in $K$. Since
$a_0(\phi \ast P_{ir})=a_0(\phi^n)=0$, (b) follows
from Lemma \ref{lem1} immediately.
\end{proof}

We next introduce a local product decomposition of $K$, which is
useful for studying analytic functions on $K$. Fix a positive
$\gamma$ in $\Gamma$, and let $K_\gamma$ be the closed subgroup of
all $x$ in $K$ such that $\chi_\gamma(x)= 1$. Then
$K_\gamma\times [0,2\pi/\gamma)$ is identified with $K$ via the map
$(y,s) \rightarrow y+e_s$. Let $\sigma_1$ be the normalized Haar
measure on $K_\gamma$. Then the probability measure
$(\gamma/2\pi)d\sigma_1 \times dt$ on
$K_\gamma\times [0,2\pi/\gamma)$ is
carried by the map to $d\sigma$ on $K$. The one-parameter group
$\{T_t\}_{t\in  \mathbb{R}}$ given by \eqref{(2)} is represented as
$$
T_t(y,s)\;=\;(y+[(t+s)\gamma/2\pi]e_{2\pi/\gamma},
t+s-[(t+s)\gamma/2\pi]2\pi/\gamma)
$$
on $K_\gamma\times [0,2\pi/\gamma)$, where $[t]$ is the
largest integer not exceeding $t$.
Define the homeomorphism $T$ on $K_\gamma$ by
$Ty\,=\,y+e_{2\pi/\gamma}$. We denote by
$\mathscr{O}(\omega,T)$ the orbit
of a point $\omega$ in $(K_\gamma, T)$, that is, the set of all
$T^n \omega$ for $n$ in $ \mathbb{Z}$. Since $\mathscr{O}(\omega,T)$ is dense
in $K_\gamma$, the discrete flow $(K_\gamma, T)$ is also a strictly
ergodic flow, on which $\sigma_1$ is the unique invariant
probability measure. Since $\Gamma$ is countable, $K_\gamma$ is
metrizable (see \cite[2.2.6]{R}).

A function $\phi$ on $K$ has the {\itshape automorphic extension}
$\phi^\sharp$ to $K_\gamma\times  \mathbb{R}$ defined by
$$
\phi^\sharp(y,t)\;=\;\phi(y+[t\gamma/2\pi]e_{2\pi/\gamma}\,,
t-[t\gamma/2\pi]2\pi/\gamma).
$$
Since a function $f$ in $H^1(dt/\pi(1+t^2))$ extends analytically to
$ \mathcal{H}$ by $f(s+ir)\,=\,(f \ast P_{ir})(s)$, we write
$$
\phi^\sharp(y,z)\;=\;(\phi^\sharp\ast P_{ir})(y,s),
\qquad z=s+ir \in  \mathcal{H}\,,
$$
for each $\phi$ in $H^1(\sigma)$. It is clear that $(\phi^\sharp\ast
P_{ir})(y,s) \,=\,(\phi\ast P_{ir})^\sharp(y,s)$ on $K_\gamma\times
 \mathbb{R}$.

The following is due to a property of Lebesgue sets.

\begin{lem}
If \/ $E_1$ is a compact subset of $K_\gamma$ with
$\sigma_1(E_1)\,>\,0$, then there is a closed subset $E$ of $E_1$
with $\sigma_1(E_1)\,=\,\sigma_1(E)$ such that
$\mathcal{O}(\omega, T)\cap E$ is dense in $E$, for
$\sigma_1-a.e.\, \omega$ in $K_\gamma$.
\label{lem3}
\end{lem}

\begin{proof}
\;\; Recall that the {\itshape metric density\/} of $E_1$ is $1$ at
$\sigma_1-a.e. \, \omega$ in $E_1$, meaning that
$$
\lim_{\delta \to 0}\, \frac{\sigma_1(E_1\cap
B(\omega,\delta))}{\sigma_1(B(\omega,\delta))} \;=\;1\,,
$$
where $B(\omega,\delta)$ is the open ball with center $\omega$ and
radius $\delta\,>\,0$. Define $E$ to be the closure of the set of
points of $E_1$ at which the metric density of $E_1$ is $1$.
Clearly, we have $\sigma_1(E_1)\,=\,\sigma_1(E)$, since  $E_1$ is closed.
If $\sigma_1(E)\,=\,1$, then $E\,=\,K_\gamma$.  Since $(K_\gamma, T)$
is strictly ergodic every orbit $\mathcal{O}(\omega, T)$ is dense in $E$.
Assume that $0\,<\,\sigma_1(E)\,<\,1$. It
follows from the ergodic theorem that there is a $\sigma_1-$null set
$N$ in $K_\gamma$ outside which
$$
\lim_{n\to \infty}\,\frac{1}{n}\sum_{j=0}^{n-1}\,
I_E(T^j \omega)\;=\;\sigma_1(E)\,,
$$
where $I_E$ denotes the characteristic function of $E$. Let
$H_\omega$ be the closure of $\mathscr{O}(\omega,T)\cap E$
in $K_\gamma$. We claim that if $E\,\not=\,H_\omega$, then
$\omega$ lies in $N$. Indeed, we see that
$\sigma_1(E\setminus H_\omega)\,>\,0$, since the metric density
of $E$ does not vanish identically on $E\setminus H_\omega$. Let
$p$ be a continuous function on $K_\gamma$ such that $0\,\le\,
p\,\le 1, \;p\,\equiv 1$ on $H_\omega$, and $\int_{K_\gamma}
 p \,d\sigma \,<\,\sigma_1(E)$.
Since $I_E(T^j \omega)\,=\,I_{H_\omega}(T^j \omega)$ for $j$ in
$ \mathbb{Z}$ and since
$(K_\gamma, T)$ is strictly ergodic, we have
$$
\limsup_{n\to \infty}\,\frac{1}{n}\sum_{j=0}^{n-1}\, I_E(T^j \omega)
\; \le\;
\lim_{n\to \infty}\,\frac{1}{n}\sum_{j=0}^{n-1}\,p(T^j\omega)
\;=\; \int_{K_\gamma} \,p\, d\sigma_1 \;<\; \sigma_1(E)
$$
by \cite[\S 4.2, Proposition 2.8]{P}. Thus $\omega$ has to lie in the
null set $N$.
\end{proof}

For each  $\phi$ in $H^\infty(\sigma)$, there is a
$\sigma_1$-null set of $K_\gamma$ outside which
$z \rightarrow \phi^\sharp(y,z)$ is analytic and uniformly bounded
on the upper half plane $ \mathcal{H}$. Recall that if a family of analytic functions
is uniformly bounded, then it forms a normal family.
The next proposition may be regarded as a strengthened form of
Lusin's theorem for analytic functions on $K$, so that it has some
interest of its own. Here we denote by $cl( \mathcal{H})$ the
closure of $ \mathcal{H}$ in $ \mathbb{R}^2$.

\begin{prop}
Let \/$\phi$ be a function in $H^\infty(\sigma)$, and let $\epsilon \,>\,0$.
Then there is a closed subset $E$ of $K_\gamma$ with $\sigma_1(E)\,>\,1-\epsilon$
having the following properties:
\begin{enumerate}
\renewcommand{\labelenumi}{(\alph{enumi})}
\item
The convolution $(\phi^\sharp\ast P_{ir})(y,t)$ is
continuous on $E\times \mathbb{R}$, for a given $r\,>\,0$.

\item  For $\sigma_1-a.e. \,\omega$ in $K_\gamma$,
the function $(\phi^\sharp\ast P_{ir})(T^j\omega,z)$ on
$(\mathscr{O}(\omega, T)\cap E)\times cl( \mathcal{H})$ extends to
$(\phi^\sharp\ast P_{ir})(y,z)$ on $E\times cl( \mathcal{H})$.
\end{enumerate}
\label{prop1}
\end{prop}

\begin{proof}\;\;
Since $\phi\ast P_{ir}$ lies in $H^\infty(\sigma)$, Lusin's theorem
asserts that there is a compact subset $F$ of $K$ with
$\sigma(F)\,>\,1-\epsilon^2$ on which $\phi\ast P_{ir}$ is continuous.
Regarding $F$ as a subset of $K_\gamma\times [0,2\pi/\gamma)$, we
choose a compact subset $E$ of $K_\gamma$ with
$\sigma_1(E)\,>\,1-\epsilon$ such that $E$ satisfies the property of
Lemma \ref{lem3} and
\begin{equation}
\frac{\gamma}{2\pi}\int^{2\pi/\gamma}_0\,I_F(y,s)\,ds\;>\;
1-\epsilon\,,\qquad y\in E\,.
\label{(5)}
\end{equation}
In addition, we assume that $z \rightarrow (\phi^\sharp\ast
P_{ir/2})(y,z)$ is analytic on $ \mathcal{H}$ and
$$
\vert (\phi^\sharp\ast P_{ir/2})(y,z)\vert\, \le \,\Vert \phi
\Vert_\infty\,, \qquad y\in E\,.
$$
Then the family
$$
\mathscr{F}\;=\;
\left\{\,(\phi^\sharp\ast P_{ir/2})(y,z)\, ; \,y \in E\right\}
$$
forms a normal family on $ \mathcal{H}$. Let $\{y_n\}$ be a
sequence in $E$ tending to $y$. Since $\mathscr{F}$ is normal, there
is a subsequence $\{y_j\}$ of $\{y_n\}$ such that $(\phi^\sharp\ast
P_{ir/2})(y_j,z)$ converges uniformly on compact subsets of
$ \mathcal{H}$ to a bounded analytic function $f(z)$ on
$ \mathcal{H}$. Let us show that $f(z)\,=\,(\phi^\sharp\ast
P_{ir/2})(y,z)$. Indeed, we observe by \eqref{(5)} that $F\cap
\,(\{y\}\times [0, 2\pi/\gamma))$ contains an infinite compact set of
the form $\{y\}\times J$. Since
$$
(\phi^\sharp\ast P_{ir})(y,t)\;=\;(\phi^\sharp\ast P_{ir/2})(y,t+ir/2)
\;=\;f(t+ir/2),\qquad t \in J\,,
$$
it follows from the uniqueness principle that
$f(z)\,=\,(\phi^\sharp\ast P_{ir/2})(y,z)$. This shows that if
$(y_n,t_n)$ tends to $(y,t)$, then $(\phi^\sharp\ast P_{ir})(y_n,t_n)$
tends to $(\phi^\sharp\ast P_{ir})(y,t)$. Thus (a) holds. We notice
that $(\phi^\sharp\ast P_{ir/2})(y,z)$ is also continuous on $E\times
 \mathcal{H}$.

On the other hand, by Lemma \ref{lem3}, $\mathscr{O}(\omega, T)\cap E$
is dense in $E$ for $\sigma_1-a.e. \,\omega$ in $K_\gamma$. Since
$(\mathscr{O}(\omega, T)\cap E)\times cl( \mathcal{H})$ is dense in
$E\times cl( \mathcal{H})$ and since $(\phi^\sharp\ast
P_{ir})(y,z)$ is continuous on $E\times cl( \mathcal{H})$,
the function $(\phi^\sharp\ast P_{ir})(T^j\omega,t)$ on
$(\mathscr{O}(\omega, T)\cap E)\times cl( \mathcal{H})$
extends to $(\phi^\sharp\ast P_{ir})(y,t)$ on $E\times
 \mathcal{H}$. Thus (b) follows immediately.
\end{proof}


We make some remarks on Proposition \ref{prop1}. Since $t \rightarrow
\phi^\sharp(y,t)$ lies in $H^\infty(dt/\pi(1+t^2))$ for each
$y$ in $E$, we see that $(\phi^\sharp\ast P_{ir})(y,t+2\pi/\gamma)\,=\,
(\phi^\sharp\ast P_{ir})(T y,t)$. Then
$E\cup TE\cup \cdots \cup T^nE$ also satisfies the
properties (a) and (b) and $\sigma_1(E\cup TE\cup \cdots \cup T^nE)$
converges to $1$, as $n\to \infty$, by the recurrence theorem (see
\cite[\S 2.3, Theorem 3.2]{P}).
However, to obtain $\phi $ itself, we need a
version of Fatou's theorem as discussed in \cite[Theorem II]{M2}.
Denote by $\mathcal{O}(x,\{T_t\}_{t\in  \mathbb{R}})$
the \textit{orbit} of $x$ in $(K, \{T_t\}_{t\in  \mathbb{R}})$.
With the notation above, when
$x=(y,s)$ in $K_\gamma\times [0,2\pi/\gamma)$, we see that
$
\mathcal{O}(x,\{T_t\}_{t\in  \mathbb{R}})\,=\,
\mathcal{O}(y, T)\times [0,2\pi/\gamma).
$
For $x$ in $K$, we say that
$t\rightarrow (\phi\ast P_{ir})(x + e_t)$ \textit{extends to}
$\phi\ast P_{ir}$ if, for each $\epsilon>0$, there
is a compact subset $F\,=\,F(\epsilon, \phi)$ of $K$ with
$\sigma(F)\,>\,1-\epsilon$ such that $\phi\ast P_{ir}$ is continuous
on $F$ and $\mathcal{O}(x, \{T_t\}_{t\in  \mathbb{R}})\cap F$ is
dense in $F$. The above proof may be modified so as to apply
to functions in $H^1(\sigma)$ as well.

The next lemma is an immediate consequence of Proposition \ref{prop1}.

\begin{lem}
Let $\phi$ be a function in $H^\infty(\sigma)$, and let $r>0$. Then
there is an invariant $\sigma-$null set $N=N(\phi)$ in $K$
outside which $t\rightarrow (\phi\ast P_{ir})(x+e_t)$
extends to $\phi\ast P_{ir}$.
\label{lem4}
\end{lem}
\begin{proof} For a given $\epsilon>0$,
let $E$ be a closed subset of $K_\gamma$ with $\sigma_1(E)\,>\,1-\epsilon$
which has the property (a) and (b) of Proposition \ref{prop1}.
Putting $F\,=\,E\times [0,2\pi/\gamma]$, we regard $F$ as a compact subset of
$K$. By (b) of Proposition \ref{prop1}, we choose an invariant null set $N'=N'(\phi)$
in $(K_\gamma, T)$ outside which $\mathscr{O}(\omega, T)\cap E$ is
dense in $E$. If we set $\, N\,=\, N'\times [0,2\pi/\gamma)$,
then the $\sigma-$null set $N$ satisfies the desired property.
\end{proof}

Let $\Omega$ be a compact metric space on which $ \mathbb{R}$ acts as a Borel
transformation group. This means that there is a one-parameter group
$\{U_t\}_{t\in  \mathbb{R}}$ of Borel isomorphisms on $\Omega$ such that
the map $(\omega,t) \rightarrow U_t\omega$ of $\Omega\times  \mathbb{R}$
to $\Omega$ is a Borel map. The pair $\fO$ is referred to a
{\itshape Borel flow}. Especially, $\fO$ is called a {\itshape
continuous flow}, if $U_t$ is a homeomorphism on $\Omega$
and the map $(\omega,t) \rightarrow U_t\omega$ is continuous on
$\Omega\times  \mathbb{R}$. We often write $\omega+t$
for the translate $U_t\omega$ of $\omega$ by $t$.
Let $\mu$ be an invariant probability measure on $\fO$ which is
{\itshape ergodic},  meaning that $\mu(E)\,=\,1$ or $0$ for each
invariant subset $E$ of $\Omega$.
A function $\phi$ in $L^1(\mu)$ is {\it analytic} if $t \rightarrow
\phi(\omega+t)$ lies in $H^1(dt/\pi(1+t^2))$ for $\mu-a.e. \, \omega$
in $\Omega$. Then the {\itshape ergodic Hardy space\/}
$H^p(\mu), 1\le p \le \infty,$
is defined to be the space of all analytic functions in $L^p(\mu)$.
It follows from \cite[Theorem I]{M1} that $\mu$ is a representing measure
for $H^\infty(\mu)$, for which $H^\infty(\mu)$ is a weak*-Dirichlet algebra
in $L^\infty(\mu)$. This fundamental result enables us to apply the
Hardy space theory based on uniform algebras, and most of the machinery
of invariant subspaces on an almost periodic flow $\fK$ can be reconstructed
(see \cite{F}, \cite{M1} and \cite{M2} for related topics). As we mentioned earlier, the $H^2_0-$spaces may be singly generated in the situation of
ergodic flows other than almost periodic flows (see \cite{T2} and \S 5 (b)).

\medskip
Let $A(x,t)$ be a cocycle on an almost periodic flow $\fK$ and define
the Borel flow $\fKT$ by
\begin{equation}
S_t(x,e^{i\theta})\;=\;(T_tx, A(x,t)e^{i\theta}),\qquad (x,e^{i\theta})
\in K\times \mathbb{T},
\label{(6)}
\end{equation}
which is called the {\itshape skew product} of $K$ and $ \mathbb{T}$ induced by
$A(x,t)$. Then $d\sigma \times d\theta/2\pi$ is an invariant probability
measure on $K\times \mathbb{T}$.
Observe that each function $f$ in $L^2(d\sigma \times d\theta/2\pi)$ is
represented as
$$
f(x,e^{i\theta})\;=\;\sum_{n=-\infty}^{\infty}\, \phi_n(x)e^{in\theta}\,,
$$
where the coefficients $\phi_n$ are in $L^2(\sigma)$. From this
fact, it follows easily that $d\sigma \times d\theta/2\pi$ is ergodic on
$\fKT$ if and only if $A(x,t)^n$ is a coboundary only for $n=0$
 (see \cite[\S 6.2]{H1} for details).

\bigskip
\section{Approximation to generators}\label{S3}

We now turn to the structure of compact group $K$, under the
assumption that $H^2_0(\sigma)$ is singly generated by
$\phi$ in $H^2_0(\sigma)$.
By multiplying by a suitable outer function, if necessary, we
can assume that $\phi$ is a function in $L^\infty(\sigma)$ with
$1\,\le\,\Vert \phi \Vert_\infty \,<\,+\infty$.
Furthermore, we also assume that $\Gamma$ is the smallest group
containing all $\lambda$ such that
$a_\lambda(\phi) \,\not= 0$, that is, the smallest group over
which Fourier series,
$$
\phi(x)\;\sim\,\sum_{\Gamma \ni \lambda >
0}\;a_\lambda(\phi)\chi_\lambda(x)\,,
$$
holds.
Similarly, denote by $\Lambda$ the smallest group containing all
$\lambda$ such that
$a_\lambda(\vert \phi \vert) \,\not= 0$.
We observe that the Fourier series  of
$$
\left(\vert\phi\vert^2+\epsilon\,\right)^{1/2}
\;=\; \exp\left\{\frac{1}{2} \,\log
\,(\phi\overline{\phi}+\epsilon)\right\}\,,\qquad \epsilon>0\,,
$$
is represented on $\Gamma$, by considering the Taylor series
of $z \rightarrow \log z $ at a large positive. This shows
that $\Lambda$ is a subgroup of $\Gamma$, since
$$
a_\lambda(\vert\phi\vert)\;=\;\lim_{\epsilon\to 0}\;\;
a_\lambda\left((\vert\phi\vert^2+\epsilon)^{1/2}\right)
$$
by \eqref{(1)}.
Since $\log \vert \phi \vert$ does not lie in $L^1(\sigma)$, the generator
$\phi$ cannot be periodic in $\fK$. Then $\Gamma$
as well as $\Lambda$ is a countable dense subgroup of
$ \mathbb{R}$, endowed with discrete topology. Let $H$ be the
annihilator of $\Lambda$, meaning that $H$ is the closed subgroup of
all $x$ in $K$ such that $\chi_\lambda(x)\,=\,1$ for all $\lambda$ in
$\Lambda$. Then the dual group of $\Lambda$ is identified with the
quotient group $K/H$ (see \cite[2.1]{R}). We denote by $\tau$ the
normalized Haar measure on $K/H$. Let $\pi$ be the canonical
homomorphism of $K$ onto $K/H$. For each $x$ in $K$, we write $\bar{x}$
for $\pi(x)=x+H$. When a function $\psi$ on $K$ is represented as
$\psi\,=\,\tilde{\psi}\circ\pi$ for a function $\tilde{\psi}$ on
$K/H$, we usually identify $\psi$ with
 $\tilde{\psi}$, so that $\psi(x)=\psi(\bar{x})$.
Then we say descriptively that $\psi$ \textit{is generated by a function
on} $K/H$. If $1\le p\le \infty$, then $L^p(\tau)$ and $H^p(\tau)$ are
subspaces of $L^p(\sigma)$ and $H^p(\sigma)$, respectively.

Since almost every $t \rightarrow \phi(x+e_t)$ is outer in
$H^\infty(dt/\pi(1+t^2))$ by Lemma \ref{lem1}, we see that
$$
-\infty\;<\;\log \vert (\phi \ast P_{ir})(x)\vert \;=\; (\log \vert
\phi\vert \ast P_{ir})(x)
$$
for a given $r\,>\,0$. Since $\log \vert \phi\vert$ is not in
$L^1(\sigma)$ and $\log \vert \phi\vert\,\le\,\Vert \phi
\Vert_\infty$, Fubini's theorem shows that
$$
\int_K \log \vert \phi\ast P_{ir}\vert \,d\sigma \; =\;
\int_K (\log \vert \phi\vert \ast P_{ir})\,d\sigma \; =\;
\int_K \log \vert \phi \vert \,d\sigma \; =\; -\infty\,.
$$
Let $g\,=\,\phi \ast P_{ir}$. Then Lemma \ref{lem1} shows that
$g$ is also a bounded generator of $H^2_0(\sigma)$.
Since $\hat{P}_{ir}(\lambda)\,=\,e^{-r \vert \lambda \vert}$
by \eqref{(4)}, we obtain $a_\lambda(g)\,=\,
a_\lambda(\phi\ast P_{ir})\,=\,a_\lambda(\phi)e^{-r\vert \lambda \vert}$,
hence $a_\lambda(\phi)\, \not=\,0$ if and only if
$a_\lambda(g)\,\not=\,0$. Thus the generator
$g$ plays the same role as $\phi$.
For $n = 1, 2, \ldots\,$,
we then denote by $\phi_n$ the outer function in $H^\infty(\sigma)$ with
$\vert \phi_n \vert \,=\,\max (1/n, \vert \phi \vert)$. Since
$-\log n \le \log \vert \phi_n \vert \le \Vert\phi \Vert_\infty$,
each $\phi_n^{-1}$ is also an outer function in $H^\infty(\tau)$.
Putting $g_n\,=\,\phi_n\ast P_{ir}$, we obtain a sequence $\{g_n\}$
of outer functions in $H^\infty(\tau)$ with
$\Vert g_n \Vert_\infty\le \Vert\phi \Vert_\infty$.
Notice that  $t \rightarrow g(x+e_t)$ and
$t \rightarrow g_n(x+e_t)$ extend analytically up to $\{ Re\, z > -r\}$.
Let us look into the relation between $g$ and $g_n$. Since
$$
\vert g_n(x)\vert\;=\;\exp\{(\log\vert\phi_n\vert \ast P_{ir})(x)\}\,,
$$
we obtain
\begin{equation}
\vert g_1(x)\vert\;\ge\;\vert g_2(x)\vert\;\ge\; \cdots\;
\ge\;\vert g_n(x)\vert \; \longrightarrow \;\vert g(x)\vert\,,
\qquad n\to\infty\,,
\label{(7)}
\end{equation}
for $\sigma-a.e.\, x\,$ in $K$. Although $g$ may not be in
$L^\infty (\tau)$, we observe that
$\vert g_n(x)\vert\,=\,\vert g_n(\bar{x})\vert $ and
$\vert g(x)\vert\,=\,\vert g\vert(\bar{x})$.
By \eqref{(7)}, it is easy to see that almost every
$ t \rightarrow \vert (g/g_n)(x+e_t)\vert$ converges pointwise to $1$
on $ \mathbb{R}$. Put $G^x_n(t) = g_n(x+e_t)$ and $G^x(t) = g(x+e_t)$.
Let $N_0$ be an invariant null set in $K$ outside which the property of
Lemma \ref{lem4} holds simultaneously for $\phi$ and all $\phi_n$.
Moreover, for $x$ in $K\setminus N_0$, we may assume $G^x_n(t)$
and $G^x(t)$ are outer functions in $H^\infty(dt/\pi(1+t^2))$.
Then the family of all analytic extensions $G^x_n(z)$
of $G^x_n(t)$ to $\{Re\, z >  -r\}$ forms a normal family, since
$\vert G^x_n(z)\vert \le \Vert\phi \Vert_\infty$.

The following lemma is crucial in our proof of the Theorem.
\begin{lem}
For a bounded generator $\phi$ of $H^2_0(\sigma)$,
let $\Lambda,\, H$ and $\tau$ be as above. Choose an $\alpha$
in $\Gamma$ with $a_\alpha(\phi)\,\not= 0$. Then
$\overline{\chi_\alpha} \phi$ is generated by
a function on $K/H$, so lies in $L^\infty(\tau)$.
Consequently, $\Gamma$ is generated by $\Lambda$ and $\alpha$.
\label{lem5}
\end{lem}
\begin{proof}\;
Let $\{\delta_k\}$ be a decreasing sequence tending to $0$.
Then there is a sequence $\{f_k\}$ in $L^1(dt)$ such that
$\hat{f_k}(\alpha)\,=\,1$, $\Vert f_k \Vert_1 \,= 1$ and
$\hat{f_k}\,=\,0$ outside $(\alpha-\delta_k, \alpha+\delta_k)$,
by modifying the function $t\rightarrow (1/\pi )\,\sin^2 t/t^2$ in $L^1(dt)$.
Since $a_\lambda(g)\,=\,a_\lambda(\phi)e^{-r \vert \lambda \vert}$, we see that
$\overline{\chi_\alpha} \phi$ lies in $L^2(\tau)$ if and only if
so does $\overline{\chi_\alpha} g$. Thus we may replace $\phi$
with $g$ in our argument. Since
$a_\lambda(g\ast f_k)\,=\,a_\lambda(g)\hat{f_k}(\lambda)$,
we observe that
$$
\Vert g\ast f_k - a_\alpha(g)\chi_\alpha \Vert^2_2\;=\; \sum_{0
<\vert \lambda\vert <\delta_k}\, \vert
a_{\alpha+\lambda}(g)\hat{f_k}(\alpha+\lambda)\vert^2 \to 0\,,
\qquad k\to \infty\,,
$$
by the Parseval theorem and that
$$
\Vert \overline{(g\ast f_k)}g -
\overline{a_\alpha(g)}(\overline{\chi_\alpha} g) \Vert_2\;\le\;
\Vert g\ast f_k - a_\alpha(g)\chi_\alpha \Vert_2 \, \Vert g
\Vert_\infty \,.
$$
From these facts, we conclude that if each $\overline{(g\ast f_k)}g$
lies in $L^\infty(\tau)$, then so does $\overline{\chi_\alpha}g$.
Since the outer function $\phi_n$ lies in $L^\infty(\tau)$,
so do $g_n$ and $g_n \ast f_k$. Then each $\overline{(g_n\ast f_k)}g_n$
lies in $L^\infty(\tau)$. Let us show that the sequence
$\{ \overline{(g_n\ast f_k)}g_n\}$ converges to $\{ \overline{(g\ast f_k)}g \}$
in $L^2(\sigma)$,
from which we obtain that $\overline{(g\ast f_k)}g$
lies in $L^\infty(\tau)$.
Indeed, in the notation above, if we fix an $x$ in $K\setminus N_0$,
there is a subsequence $\{g_m\}$ of $\{g_n\}$ such that $\{G_m^x(t)\}$
converges pointwise to $e^{i\gamma}G^{x}(t)$ in $H^\infty(dt/\pi(1+t^2))$
with $0\le\gamma<2\pi$, where $\gamma$ depends on $x$ and $\{g_m\}$.
This implies that
$$
\overline{(g_m\ast f_k)}(x+e_t) \rightarrow
e^{-i\gamma}\overline{(g\ast f_k)}(x+e_t),
\qquad m\to\infty,
$$
pointwise in $L^\infty(dt/\pi(1+t^2))$.
Note that every subsequence of $\{g_n\}$ contains such a subsequence
$\{g_m\}$. Since $e^{-i\gamma}e^{i\gamma}\,=\,1$, the sequence
$\{g_n\}$ itself satisfies
$$
\overline{(g_n\ast f_k)}\, g_n(x+e_t) \rightarrow \overline{(g\ast
f_k)}\, g(x+e_t)\,,
\qquad n\to\infty,
$$
pointwise in $L^\infty(dt/\pi(1+t^2))$. Since
$$
\Vert \overline{(g_n\ast f_k)}\, g_n \Vert_\infty \;\le \; \Vert g_n
\Vert_\infty^2 \Vert f_k\Vert_1 \;\le \;
\Vert \phi \Vert_\infty^2 \Vert f_k\Vert_1,
$$
it follows from the bounded convergence theorem that
$$
\Vert \overline{(g_n\ast f_k)}\, g_n - \overline{(g \ast f_k)}\, g
\Vert_2 \rightarrow 0\,, \qquad n\to\infty,
$$
so that $\overline{(g \ast f_k)}\, g$ lies in $L^\infty(\tau)$. Therefore,
$\overline{\chi_\alpha}\, g$ as well as $\overline{\chi_\alpha}\, \phi$
is generated by a function on $K/H$. On the other hand, by
the property of $\Gamma$, each element in $\Gamma$
has the form $\lambda + n\alpha$ for $\lambda$ in $\Lambda$
and $n$ in $ \mathbb{Z}$, thus the proof is complete.
\end{proof}

Recall that $K/H$ coincides with the dual group of $\Lambda$.
Let $\alpha$ be as in Lemma \ref{lem5} and let $C(\bar x,t)$ be the
trivial cocycle on $K/H$ defined by $C(\bar x,t)=\exp (i \alpha t)$.
Since $\alpha$ is positive, $C(\bar x,t)$ is an analytic cocycle.
We denote by $(K/H\times \mathbb{T}, \{S_t\}_{t\in  \mathbb{R}})$
the skew product of $K/H$ and $ \mathbb{T}$ induced by $C(\bar x,t)$,
which is the continuous flow obtained by
\begin{equation*}
S_t(\bar x,e^{i\theta})\;=\;(T_t\,\bar x, C(\bar x,t)e^{i\theta}),
\qquad (\bar x,e^{i\theta}) \in K/H\times \mathbb{T}\,.
\end{equation*}
Then $d\tau \times d\theta/2\pi$ is the invariant probability
measure on $K/H\times \mathbb{T}$ (see the end of the preceding section).
Let us represent the generator $g$ and all the limits of subsequences of $\{g_n\}$
on $K/H\times \mathbb{T}$, which is the smallest product group with
such property. Each function $\psi$ on $K/H$ extends naturally to the
one on $K/H\times \mathbb{T}$ by setting $\psi(\bar x,e^{i\theta})\,=\,\psi(\bar x)$.
Since $\vert g \vert$ and $g_n$ are functions on $K/H$, they belong to
$L^\infty(d\tau\times d\theta/2\pi)$.

With the above notation, we fix a $w$ in $K\setminus N_0$. Since
$G_n^{w}(t)$ and $G^{w}(t)$ are outer functions in $H^2(dt/\pi(1+t^2))$
which extend analytically to $\{Re \,z >  -r\}$, we may assume that
$G_n^{w}(t)$ converges pointwise to $G^{w}(t)$ on $ \mathbb{R}$,
by multiplying each $g_n$ by a suitable constant of modulus one.
By regarding Lemma \ref{lem4}, the functions $G_n^{w}(t)$ and $G^{w}(t)$
extend to $g_n$ and $g$, respectively.
However, we obtain the following:
\begin{lem}
\; For $\sigma-a.e. \,x$ in $K$, $G_n^x(t)$ never converges pointwise
on $\mathbb{R}$.
Consequently, we find two subsequences $\{g_m\}$ and
$\{g_k\}$ of $\{g_n\}$ such that $G_m^x(t)$ and $G_k^x(t)$
converge to $e^{i\beta}G^{x}(t)$ and
 $e^{i\gamma}G^{x}(t)$ with $0\le \beta < \gamma <2\pi$, respectively.
\label{lem6}
\end{lem}
\begin{proof}
\;\; Since $1/n\le \vert g_n(x)\vert \le \Vert\phi \Vert_\infty$, each
$g_n^{-1}$ is also an outer function in $H^\infty(\sigma)$.
This implies that almost every $t\rightarrow (g/g_n)(x+e_t)$
is an outer function in $H^\infty(dt/\pi(1+t^2))$.
Furthermore, since
$$
a_0(g/g_n)\;=\; \int_K\,g/g_n\,d\sigma
\;=\;\int_K\,g\,d\sigma\int_K\,g_n^{-1}\,d\sigma\;=\;0\,,
$$
Lemma \ref{lem1} assures that each $g/g_n$ is also a single
generator of $H^2_0(\sigma)$.

Denote by $F$ the invariant set of all $x$ in $K$ for which
 $\{G_n^x(t) \}$ itself converges.
Suppose that $F$ has positive measure. By \eqref{(7)} and
the ergodic theorem, $(g/g_n)(x)$ converges to an
 invariant function on $F$, so to a constant of modulus
 one on $K$. Then the bounded convergence theorem
shows that $a_0(g/g_n)\not=0$ for large $n$.
Such $g/g_n$ cannot be a single generator of $H^2_0(\sigma)$,
which contradicts the above observation.
\end{proof}

Let us mention a few remarks derived from Lemma \ref{lem6}. When
$0\le \beta < 2\pi$, $\mathcal{Z}(\beta)$ denotes the subgroup
of $ \mathbb{T}$ generated by $e^{i\beta}$, that is,
$$
\mathcal{Z}(\beta)\;=\;
\left\{e^{ij\beta}\;;\;j\in  \mathbb{Z}\,\right\}\,.
$$
If $\beta/2\pi$ is rational, then the order of
$\mathcal{Z}(\beta)$ is finite.
Fix two points $w$ and $x$ in $K\setminus N_0$. We assume
by Lemma \ref{lem6} that a subsequence $\{g_k\}$ of
$\{g_n\}$ satisfies that $G_k^w(t)$
and $G_k^x(t)$ converge respectively to $e^{ij\beta}G^{w}(t)$ and
$e^{i(j+1)\beta}G^{x}(t)$ for $j$ in $ \mathbb{Z}$,
by multiplying each $g_k$ by a suitable constant of modulus one.
Denote by $\mathcal{O}(\bar{w})$ the orbit
 $\mathcal{O}(\bar{w}, \{T_t\}_{t\in {\bf R}})$ of
$\bar{w}$ in $(K/H, \{T_t\}_{t\in  \mathbb{R}})$.
Then $g$ is determined naturally on
$\mathcal{O}(\bar{w})\times \mathcal{Z}(\beta)$ and
$\mathcal{O}(\bar{x})\times \mathcal{Z}(\beta)$ to represent the
limits of the subsequence $\{g_k\}$ of  $\{g_n\}$ on them.
For each $m$ in $ \mathbb{Z}$, we see also that every limit of
$\{g_k^m\}$ is represented  on these product subsets.

If $\ell$ is a positive integer, then $g^\ell$ as well as $\phi^\ell$
is also a bounded generator of $H^2_0(\sigma)$ by Lemma \ref{lem2}.
We choose an invariant null set
$N(\ell)$ including $N_0$ outside which
a subsequence $\{G_j^x(t)^\ell\}$ of $\{G_n^x(t)^\ell\}$
converges to $e^{i\gamma}G^{x}(t)^\ell$ with
$0<\gamma < 2\pi$.  Define the invariant null set
 $N_1$ by $N_1\;=\;\cup_{\ell=1}^{\infty}\,N(\ell)\,$.
When $\ell\,=\, m\, !$,
we take again a subsequence $\{G_k^x(t)\}$ of
$\{G_j^x(t)\}$ converging  to $e^{i\beta(m)}G^{x}(t)$ with
$e^{i\beta(m) \ell}\,=\,e^{i\gamma}$. Then the order of
$\mathcal{Z}(\beta(m))$ is larger than $m$, so
$\cup_{m=1}^{\infty}\,\mathcal{Z}(\beta(m))$ is dense in $ \mathbb{T}$.
Therefore, to represent $g$ and all the limits of subsequences of
$\{g_n\}$ on each orbit, the product group $K/H\times \mathbb{T}$
is the smallest one. Let us explain the meaning more precisely.
Under the assumption of Lemma \ref{lem5}, we put
$h_\alpha=\overline{\chi_\alpha}g$. Then
 $h_\alpha$ lies in $L^2(\tau)$.
Define the group character $\mathcal{P}_\alpha$ of $K/H\times \mathbb{T}$
by the projection $\mathcal{P}_\alpha(\bar x,e^{i\theta})=e^{i\theta}$. Since
$$
(h_\alpha \mathcal{P}_\alpha)(S_t(\bar x,e^{i\theta}))
\,=\,h_\alpha(\bar x +e_t) C(\bar x,t)e^{i\theta}
\,=\,h_\alpha(\bar x +e_t) e^{i\alpha t}e^{i\theta},
$$
the function
$t\rightarrow (h_\alpha \mathcal{P}_\alpha)(S_t(\bar x,e^{i\theta}))$
is an outer function in $H^\infty(dt/\pi(1+t^2))$ for
$d\tau \times d\theta/2\pi-a.e.\, (\bar x,e^{i\theta})$
in $K/H\times \mathbb{T}$. Then the outer function $G^x(t)$ equals
$t\rightarrow (h_\alpha \mathcal{P}_\alpha)(S_t(\bar x,e^{i\theta}))$
for some $\theta$ with $0\le \theta < 2\pi$.
In order to represent consistently all kinds of limits of subsequences
$\{G_k^x(t)\}$, we require the family of all outer functions $t\rightarrow (h_\alpha \mathcal{P}_\alpha)(S_t(\bar x,e^{i\theta}))$ with $0\le \theta < 2\pi$.

\begin{lem}
Let $\Gamma$ and $\Lambda$ be as above. Then $\Lambda$ cannot be
equal to  $\Gamma$.
\label{lem7}
\end{lem}

\begin{proof}
\;\; Let $\alpha$ be as in Lemma \ref{lem5}. Then
$\alpha$ lies in $\Lambda$ if and only if  $\Lambda \,=\, \Gamma$.
We suppose, on the contrary, that  $\alpha$ lies in $\Lambda$.
Since $K/H\,=\,K$, let us consider
the skew product $(K\times \mathbb{T}, \{S_t\}_{t\in  \mathbb{R}})$ of $K$
and $ \mathbb{T}$ induced by the cocycle $C(x,t)\,=\, e^{i \alpha t}$.
We use freely the notation above.
Since
$$
\mathcal{F}(x,e^{i\theta})\;=\;
(\overline{\chi_\alpha}\, \mathcal{P}_\alpha)(x,e^{i\theta})
\,,\qquad (x,e^{i\theta}) \in K\times \mathbb{T}\,,
$$
is an invariant function that is not constant,
$d\sigma \times d\theta/2\pi$ is not an ergodic measure on
$(K\times \mathbb{T}, \{S_t\}_{t\in  \mathbb{R}})$.
Now $K$ is represented as the local product
decomposition $K_\alpha\times[\,0,2\pi/\alpha)$,
in which $K_\alpha$ is the closed subgroup of all $x$ in $K$
such that $\chi_\alpha(x)\,=\,1$.
If we put
$$
\mathcal{G}(x,e^{i\theta})\;=\;
h_\alpha(x) \, \mathcal{P}_\alpha(x,e^{i\theta})\,,
\qquad (x,e^{i\theta})\,\in K\times \mathbb{T},
$$
then, for each $x=(y,s)$ in $K_\alpha\times[0,2\pi/\alpha)$,
the equation
\begin{equation}
\mathcal{G}(S_t(x,e^{i\theta}))\,
=\,e^{i(\theta+\alpha t)}h_\alpha(x+e_t)\,
=\,e^{i(\theta-\alpha s)}g(x+e_t)
\label{(8)}
\end{equation}
holds, since $e^{i(\theta+\alpha t)}\overline{\chi_\alpha}(y+e_s+e_t)
\,=\,e^{i(\theta-\alpha s)}$ and $h_\alpha=\overline{\chi_\alpha}g$.
By regarding $ \mathbb{T}$ as the interval $[\,0,2\pi/\alpha)$,
$K\times \mathbb{T}$ is identified with
$K_\alpha\times[\,0,2\pi/\alpha)\times[\,0,2\pi/\alpha)$.
Let $E$ be the subset of $K\times \mathbb{T}$ defined by
$$
E\,=\,K_\alpha\times\{(s,s)\,; 0\le s <2\pi/\alpha\}.
$$
Then $E$ is a closed invariant set in
$(K\times \mathbb{T}, \{S_t\}_{t\in  \mathbb{R}})$, for which
$\fK$ is isomorphic to $(E, \{S_t\}_{t\in  \mathbb{R}})$ via the
map $(y,s) \to (y,s,s)$. We see also that the ergodic measure
$d\sigma$ is carried to $(\alpha/2\pi)d\sigma_1 \times ds $
on $E$ by this map, where $\sigma_1$ is the normalized Haar
measure on $K_\alpha$.
We regard $g_n, \,g $ and $h_\alpha$ as the functions on
$(K\times \mathbb{T}, \{S_t\}_{t\in  \mathbb{R}})$.  Recall
that almost every $G_n^x(t)$ and $G^{x}(t)$ are outer
functions in $H^\infty(dt/\pi(1+t^2))$.

Let $x$ be in $K \setminus N_1$ and let $\{g_k\}$ be a subsequence
 of $\{g_n\}$ such that $G_k^x(t)$
converges pointwise to $t \rightarrow
e^{i\alpha \beta}e^{i\alpha t}h_\alpha(x+e_t)$
with $0\le \beta <2\pi/\alpha$.
Notice that $t \rightarrow e^{i\alpha \beta}e^{i\alpha t}h_\alpha(x+e_t)$ is an outer
function in  $H^\infty(dt/\pi(1+t^2))$ and that
$\vert h_\alpha(x+e_t)\vert\,=\,\vert g(x+e_t)\vert$.
Let $x=(y,s)$ in $K_\alpha\times[\,0,2\pi/\alpha)$ as above.
Since $x$ may be
replaced by any point in the orbit $\mathcal{O}(x)$ of $x$,
we consider $x$ as a function of $s$ on $[\,0,2\pi/\alpha)$.
It follows from
\eqref{(8)} that
$$
e^{i\alpha \beta}e^{i\alpha t }h_\alpha(y+e_s+e_t)
\;=\;
e^{i\alpha(\beta-s)}\mathcal{G}(S_t(y+e_s,e^{i\alpha s}))\,,
\quad (s,t) \in [\,0,2\pi/\alpha)\times  \mathbb{R}.
$$
Putting $t=0$ and replacing $y$ with $y+e_{[s\alpha/2\pi]}$,
if necessary, we observe that
$$
e^{i\alpha (\beta-s)}\mathcal{G}(y+e_s,e^{i\alpha s})
\;=\;e^{i\alpha \beta} e^{-i\alpha s}G^{y}(s)\,,
\qquad s \in  \mathbb{R}.
$$
This shows that  $G_k^y(s)$
converges pointwise to $s \rightarrow
e^{i\alpha \beta}(\overline{\chi_\alpha}\,g)(y+e_s)$, which cannot
be an outer function in  $H^\infty(dt/\pi(1+t^2))$.  Hence
any subsequence of $\{G_n^x(t)\} $ cannot converge to an outer
function in $H^\infty(dt/\pi(1+t^2))$ for $\sigma-a.e. \,x$ in $K$.
Thus we have a contradiction.
\end{proof}

\medskip
In view of Lemma \ref{lem7}, we know that there are two possibilities
in relation to $\alpha$ and $\Lambda$. Either $n\alpha$ lies
in $\Lambda$ only for $n=0$ or $\ell\alpha$ lies in
 $\Lambda$ for an integer  $\ell \ge 2$.
We claim that the latter case cannot occur, meaning that $\alpha$ is
independent to $\Lambda$.

\begin{lem}
Let $\Lambda$, $H$ and $\/\alpha$ be as above. Then $n\alpha$
lies in $\Lambda$ if and only if $n=0$ in $ \mathbb{Z}$.
Consequently, $H$ is isomorphic to $ \mathbb{T}$, so that
$K$ and $\,d\sigma$ are identified with $K/H\times  \mathbb{T}$
and $d\tau\times d\theta/2\pi$, respectively.
\label{lem8}
\end{lem}

\begin{proof}
\;\; Suppose that  $\ell\alpha$ lies in $\Lambda$
for some  $\ell \ge 2$. By Lemma \ref{lem2},
$\phi^\ell$ is also a bounded generator of $H^2_0(\sigma)$.
It follows from  Lemma \ref{lem6} that $\chi_{\ell\alpha}$ and $(\overline{\chi_\alpha}\phi)^\ell$
lie in $L^2(\tau)$, so does $\phi^\ell$ itself. Let  $\Gamma_\ell$
and  $\Lambda_\ell$ be the smallest groups determined by the nonzero
Fourier coefficients of $\phi^\ell$ and $\vert\phi^\ell\vert$ as above.
Then they both are subgroups of $\Lambda$.
On the other hand, since
$$
a_\lambda(\vert\phi\vert)\;=\;\lim_{\epsilon\to +0}\;
a_\lambda\left((\vert\phi\vert^\ell+\epsilon)^{1/\ell}\right)\,,
$$
each $\lambda$ in $\Lambda$ with
$a_\lambda(\vert\phi\vert)\,\not=\, 0$ lies in $\Lambda_\ell$.
This implies that $\Lambda\,=\,\Lambda_\ell \,=\,\Gamma_\ell$.
By replacing $\phi$ with $\phi^\ell$ in Lemma \ref{lem7}, this gives
a contradiction.
Thus $n\alpha$ lies in $\Lambda$ if and only if $n=0$.

Since $C(\bar x,t)^n$ is a coboundary only
for $n=0$, the measure $d\tau \times d\theta/2\pi$ is ergodic on
$(K/H\times \mathbb{T}, \{S_t\}_{t\in  \mathbb{R}})$. Define the
isomorphism of $\Lambda\times  \mathbb{Z}$ onto $\Gamma$ by
$$
\varrho(\lambda, n) \;=\; \lambda+n\alpha\,,\quad\quad
(\lambda, n)\in \Lambda\times  \mathbb{Z}.
$$
Then the conjugate map $\varrho^*$ of $\varrho$ is given by
$\varrho^*(x)=
(\bar x,e^{i\theta})$ on $K$, where $\chi_\alpha(x)\,=\,e^{i\theta}$.
Indeed, we observe that
$$
\chi_\lambda(\bar x) e^{in\theta}\;=\;\langle (\lambda, n),(\bar x,e^{i\theta})\rangle
\;=\;\chi_{\lambda+n\alpha}(x)\;=\;\chi_\lambda(\bar x)\chi_\alpha(x)^n\,,
$$
for each $(\lambda, n)$ in $\Lambda\times  \mathbb{Z}$.
Via the map $\varrho^*$\,, $K$ is identified with
$K/H\times \mathbb{T}$, and $d\tau \times d\theta/2\pi$ is
carried by the map to $d\sigma$ on $K$.
\end{proof}

We notice that the annihilator $H$ of $\Lambda$ is isomorphic to
 $ \mathbb{T}$, and $\vert g(x)\vert$ as well as $\vert \phi(x)\vert$
is constant on almost every coset $\bar{x}=x+H$ in $K/H$.

\bigskip
\section{Contradiction to existence}

We may now offer our proof of the main result stated in Section 1.

\medskip
\noindent
{\itshape Proof of the Theorem.\;}
Suppose, on the contrary, that a bounded function $\phi$
generates $H^2_0(\sigma)$. Let $\Gamma$ and $\Lambda$
be the dense subgroups of $ \mathbb{R}$ defined as in Section 3
with respect to $\phi$ and $\vert\phi \vert$, respectively.
Choose an $\alpha$ in $\Gamma$ with $a_\alpha(\phi) \not= 0$.
It follows from Lemma \ref{lem8}
that $\alpha$ is independent of $\Lambda$ and $\Gamma$ is
generated by $\alpha$ and $\Lambda$.
Let $0<\beta <1$.  Since the function
$$
(1+\beta\chi_\alpha)^{-1}\;=\; \sum_{k=0}^{\infty}\,
(-\beta)^k\chi_{k\alpha}
$$
lies in $H^\infty(\sigma)$, $(1+\beta\chi_\alpha)^2$ is an
outer function in $H^\infty(\sigma)$. Define
$\phi_1\,=\,(1+\beta\chi_\alpha)^2\,\phi$\,.
In view of Lemma \ref{lem1},  $\phi_1$ is also a bounded
generator of $H^2_0(\sigma)$. As above, let $\Gamma_1$ and
$\Lambda_1$ be the smallest groups determined by the nonzero
Fourier coefficients of $\phi_1$ and $\vert\phi_1\vert$, respectively.
Notice that $\Gamma_1$ is a subgroup of $\Gamma$.
We claim that the generator $\phi_1$ cannot satisfy the
property of Lemma \ref{lem7}.
Indeed, since
$
\vert \phi_1 \vert\,=\,(1+\beta^2+\beta\overline{\chi_\alpha}
+\beta\chi_\alpha)\, \vert \phi \vert\,,
$
we obtain by \eqref{(1)} that
$$
a_\lambda(\vert \phi_1 \vert)\;=\;(1+\beta^2)a_\lambda(\vert \phi \vert)
+\beta a_{\lambda+\alpha}(\vert \phi \vert)+ \beta
a_{\lambda-\alpha}(\vert \phi \vert).
$$
Since $\alpha$ does not lie in $\Lambda$, if $\lambda$ is in
$\Lambda$, then $a_{\lambda+\alpha}(\vert \phi \vert)\,=\,
a_{\lambda-\alpha}(\vert \phi \vert)\,=0$. Then we have
$$
a_\lambda(\vert \phi_1 \vert)\,=\,
(1+\beta^2)a_\lambda(\vert \phi \vert) \quad \text{and} \quad
a_{\lambda+\alpha}(\vert \phi_1 \vert)\,=\,
\beta a_\lambda(\vert \phi \vert),
$$
for each $\lambda$ in $\Lambda$.
These facts imply that $\Lambda_1$ contains $\Lambda$ and
$\alpha$, so that $\Gamma=\Lambda_1=\Gamma_1$, which contradicts
Lemma \ref{lem7}.
\hfill $\square$

\medskip
The next proof is of independent interest, because it suggests
that our Theorem is regarded essentially as the converse
to Corollary \ref{cor1}.

\medskip
\noindent
{\itshape Proof of Corollary \ref{cor1}.\,}
We consider the case where the cocycle $C(x,t)$
of $\,\mathfrak{M}$ has the form $C(x,t)\,=\,e^{i\alpha t}$.
Then $\mathfrak{M}_-$ is the space of all $\psi$ in
$L^2(\sigma)$ satisfying that
$$
\psi(x)\;\sim\,\sum_{\Gamma \ni \lambda \,>\, -\alpha}\;
a_\lambda(\psi)\chi_\lambda(x)\,.
$$
Suppose that $\mathfrak{M}_-$ has a generator $\phi$.
Then $\log\vert\phi\vert$ does not lie in $L^1(\sigma)$
and we may assume that $\phi$ is bounded.
If $\ell \alpha$ is in $\Gamma$ for a positive integer $\ell$,
then the bounded function $(\chi_\alpha \phi)^\ell$ is a
single generator of $H^2_0(\sigma)$ by Lemma \ref{lem1},
which is contrary to Theorem. We next consider the case that
$$
\alpha \in  \mathbb{R}\setminus \bigcup^\infty_{n=1}\,(1/n)\Gamma\,.
$$
Since ${C(x,t)}^n$ is a coboundary only for $n=0$, the measure
$d\sigma \times d\theta/2\pi$ is ergodic on the skew product
$(K\times \mathbb{T}, \{S_t\}_{t\in  \mathbb{R}})$ induced by
$C(y,t)$, that is,
$$
S_t(x,e^{i\theta})\;=\;(x+e_t\,, \,e^{i\alpha t}e^{i\theta}),
\qquad (x,e^{i\theta})\in K\times \mathbb{T}.
$$
Let $\Gamma_1$ be the discrete group generated by
$\Gamma$ and $\alpha$, and let $K_1$ be the dual group of $\Gamma_1$.
Since $\varrho(\lambda, n)\,=\,\lambda + \alpha n$ is an isomorphism
of $\Gamma\times \mathbb{Z}$ onto $\Gamma_1$, the almost periodic
flow on $K_1$ is identified with
$(K\times \mathbb{T}, \{S_t\}_{t\in  \mathbb{R}})$. Then, via the dual
map $\varrho^\ast$ of $\varrho$, the normalized Haar measure
$d\mu$ on $K_1$ is identified with $d\sigma \times d\theta/2\pi$.
Define the function $\phi_1$ in $L^2(\mu)$ by
$\phi_1(x,e^{i\theta})\,=\,\phi(x)e^{i\theta}$.
Since $\log\vert\phi\vert$ does not lie in $L^1(\sigma)$,
neither does $\log\vert\phi_1\vert$ in $L^1(\mu)$. Since
$t \rightarrow \phi_1\circ S_t(x,e^{i\theta})$
is outer in $H^2(dt/\pi(1+t^2))$ for $\mu - a.e.\,
(x,e^{i\theta})$ in $K\times \mathbb{T}$,
Lemma \ref{lem1} implies that
$\phi_1$ is a single generator of
$H^2_0(\mu)$, which contradicts our Theorem.
\hfill $\square$

\medskip
\noindent
{\itshape Proof of Corollary \ref{cor2}.\;}
Denote by $C(x,t)$ the real cocycle of $\mathfrak{M}$.
Suppose that $\mathfrak{M}_-$ has a generator $\phi$, for which
$\log \vert\phi\vert$ does not lie in $L^1(\sigma)$.
It follows from Lemma \ref{lem1} that almost every
$t \rightarrow C(x,t)\phi(x+e_t)$ is outer in $H^2(dt/\pi(1+t^2))$.
We may assume that $\phi$ is bounded. Since
$C(x,t)^2\,\equiv \,1$, $\phi^2$
is a single generator of $H^2_0(\sigma)$ by Lemma \ref{lem1},
which contradicts our Theorem.
\hfill $\square$

\medskip
By the same way as above, we may show that if $C(x,t)$ takes only
finite values, then $\mathfrak{M}_-$ cannot be singly generated.
Indeed, by the cocycle identity,
the set of values of $C(x,t)$ forms a group of order $k$,
$$
\mathcal{Z}(2\pi/k)\;=\;\left\{e^{i2\pi j/k}\;;\;j= 0,\ldots, k-1 \right\}\,.
$$
Then if $\phi$ generates $\mathfrak{M}_-$, then $\phi^k$ is
a generator of $H^2_0(\sigma)$.

\medskip
Let $\mathfrak{M}$ be the normalized simply invariant subspace of
$L^2(\sigma)$ with cocycle $A(x,t)$. Recall that $\psi$ lies in
$\mathfrak{M}$ if and only if almost every $t \rightarrow
A(x,t)\psi(x+e_t)$ lies in $H^2(dt/\pi(1+t^2))$. Denote by
$\widetilde{\mathfrak{M}}$ the invariant subspace with cocycle
$\overline{A(x,t)}$ (as discussed in \cite[\S 3.2]{H1}).
To prove Corollary \ref{cor3}. we need the following:

\begin{lem}
Let $\/\mathfrak{M}$ and $\/\widetilde{\mathfrak{M}}$ be as above.
If $\/\mathfrak{M}$ is singly generated, then
$(\widetilde{\mathfrak{M}})_-$ cannot be singly generated.
\label{lem9}
\end{lem}
\begin{proof}
Since $A(x,t)\cdot\overline{A(x,t)}\equiv 1$, $H^2_0(\sigma)$ is
the smallest subspace of $L^2(\sigma)$ containing all
$\psi_1\psi_2$ with $\psi_1$ in
$\mathfrak{M}\cap L^\infty(\sigma)$ and $\psi_2$ in
$(\widetilde{\mathfrak{M}})_-\cap L^\infty(\sigma)$ (see
\cite[\S 3.2, Theorem 20]{H1}). Suppose that $(\widetilde{\mathfrak{M}})_-$
is singly generated. Then Lemma \ref{lem1} shows that there are bounded
single generators $\phi_1$ and $\phi_2$ of $\mathfrak{M}$ and
$(\widetilde{\mathfrak{M}})_-$, respectively. Thus $\phi_1 \phi_2$
is a single generator of $H^2_0(\sigma)$, which contradicts our Theorem.
\end{proof}

\medskip
\noindent
{\itshape Proof of Corollary \ref{cor3}.\/}
(a)\; Let $\mathfrak{M}$ be a simply invariant subspace with
nontrivial cocycle $A(x,t)$. It follows from \cite{HT} that
$\mathfrak{M}$ is singly generated if and only if $A(x,t)$ is
cohomologous to a singular cocycle. On the other hand, by
\cite[\S 4.6, Theorem 26]{H1}, every cocycle is cohomologous to a
Blaschke cocycle. By virtue of Lemma \ref{lem9}, we obtain
easily a desired Blaschke cocycle.

\medskip
(b)\; From Lemma \ref{lem9},
we choose a Blaschke cocycle $B(x,t)$ such that
the invariant subspace $\mathfrak{N}$ having the
cocycle $\overline{B(x,t)}$ is not singly generated.
We claim that $B(x,t)$ satisfies the desired property.
Suppose, on the contrary, that some function $\psi$ in
$H^2(\sigma)$ has exactly the same zeros as
$B(x,t)$. By multiplying by a suitable outer function, we
assume that $\psi$ is bounded.
Then $\psi$ generates the invariant subspace with
cocycle  $\overline{B(x,t)} \overline{S(x,t)}$,
where $S(x,t)$ is the singular
cocycle determined by the inner part of
$t\rightarrow \overline{B(x,t)}\psi(x+e_t)$ in
$H^2(dt/\pi(1+t^2))$.
On the other hand, it follows from \cite{HT} and
Lemma \ref{lem1} that there is a function $h$ in
 $L^2(\sigma)$ such that almost every
$t\rightarrow S(x,t)h(x+e_t)$ is outer in
$H^2(dt/\pi(1+t^2))$. Observe that
$$
(h\psi)(x+e_t)\;=\; B(x,t)\cdot S(x,t)h(x+e_t)\cdot
\overline{B(x,t)} \overline{S(x,t)} \psi(x+e_t)\,.
$$
Since the inner part of
$t\rightarrow (h\psi)(x+e_t)$ is $t\rightarrow B(x,t)$, the subspace
$\mathfrak{N}$ is singly generated by $h\psi$, thus we have a contradiction.
\hfill $\square$

\medskip
In the proof of (b) above, if the singular cocycle $S(x,t)$ is a
coboundary, then $h$ is taken as a unitary function, otherwise
$\log \vert h\vert$ does not lie in $L^1(\sigma)$.

\bigskip
\section{Remarks}
(a)\quad It is sometimes useful to study the spectral measures
associated with invariant subspaces.  Let $\mathfrak{M}$ be a simply
invariant subspace of $L^2(\sigma)$ and put
$$
\mathfrak{M}_{\lambda}\;=\;
\bigwedge_{\lambda\ge\nu}\;\chi_\nu\,\mathfrak{M}.
$$
for each $\lambda$ in $ \mathbb{R}$.
Denote by $P_{\lambda}$ the orthogonal projection of $L^2(\sigma)$
onto $\mathfrak{M}_{\lambda}$.
By the property that
$$
\displaystyle \bigwedge_{-\infty<\lambda<\infty}\mathfrak{M}_\lambda
\;=\;\{0\} \qquad\text{and}\qquad \displaystyle
\bigvee_{-\infty<\lambda<\infty}\mathfrak{M}_\lambda
\;=\;L^2(\sigma)\,,
$$
we obtain the contuinity of the spectral resolution of identity
$\{I-P_{\lambda}\}_{\lambda\in  \mathbb{R}}$ on $L^2(\sigma)$, where
$I$ is the identity map on $L^2(\sigma)$. Let $A(x,t)$ be the
cocycle of $\mathfrak{M}$. By Stone's theorem, a unitary group
$\{V_t\}_{t\in {\bf R}}$ on $L^2(\sigma)$ is defined as
$$
V_t\,\phi(x)\;=\;A(x,t)T_t\,\phi(x)\;=\;-\int_{-\infty}^{\infty}\,
e^{i\lambda t}\,dP_{\lambda}\phi(x)\,,\qquad \phi \in L^2(\sigma)\,,
$$
where $T_t\,\phi(x)= \phi(x+e_t)$. For a nonzero function $\phi$ in
$L^2(\sigma)$, $-d(P_{\lambda}\phi, \phi)$ is a finite positive
measure on $ \mathbb{R}$. On almost periodic flows, by comparing with
Lebesgue measure $d\lambda$, the type of such measures is uniquely
determined. We then say that each of $\mathfrak{M}, A(x,t)$
and $\{V_t\}_{t\in  \mathbb{R}}$ is of \textit{absolutely
continuous}, or  \textit{singular continuous}, or \textit{discrete}
type (as discussed in \cite[\S 2.4]{H1}). This fact plays an important
role to classify invariant subspaces in this special context.
It is easy to observe that $A(x,t)$ and ${\overline{A(x,t)}}$
have the same spectral type, so the following is an immediate
consequence of Lemma \ref{lem9}.

\begin{prop}
\label{prop2}
There is a simply invariant subspace of $L^2(\sigma)$ of either
absolutely continuous or singular continuous type which has no single
generator.
\end{prop}

Let $w$ be a nonnegative function in $L^2(\sigma)$ satisfying \eqref{(3)},
while $\log w$ does not lie in $L^1(\sigma)$. We know that a
cocycle is trivial if and only if it is of discrete type (see
\cite[\S 2.4, Theorem 15 ]{H1}). It follows from Corollary \ref{cor1} that
the type of $\mathfrak{M}[w]$ has to be continuous. However, we have no
idea to decide what kind of continuous spectrum $\mathfrak{M}[w]$ may have.

\medskip
(b)\quad Using a suitable cocycle, we may construct a skew product
on which the $H^2_0-$space is singly generated. Indeed, let $w$ be a bounded
function as above and let $A(x,t)$ be the cocycle of $\mathfrak{M}[w]$.
By Lemma \ref{lem1} we see that almost every
$t \rightarrow A(x,t) w(x+e_t)$ is outer in $H^2(dt/\pi(1+t^2))$.
Denote by $\fKT$ the skew product induced by $A(x,t)$. If $A(x,t)^n, n\ge 1,$
is a coboundary $\overline{q(x)}q(x+e_t)$ with unitary function $q$ on $K$,
then $qw^n$ is a single generator of $H^2_0(\sigma)$. It then
follows from Theorem that $A(x,t)^n$ is a coboundary only for
$n=0$. Hence $d\mu \,=\,d\sigma \times d\theta/2\pi$ is an
ergodic measure on $\fKT$. If we set
$$
\phi(x,e^{i\theta})\;=\;w(x)e^{i\theta}\,,
\qquad (x,e^{i\theta}) \in K\times \mathbb{T}\,,
$$
then $\phi$ is a single generator of $H^2_0(\mu)$, since $\log |\phi|$
does not lie in $L^1(\mu)$ and almost every
$t \rightarrow \phi(S_t(x,e^{i\theta}))$
is outer in $H^2(dt/\pi(1+t^2))$
(see \cite{T2} for another construction).

\medskip
(c)\quad We have a bit of information on the distribution
of zeros of  functions in $H^2(\sigma)$ which are
connected with Dirichlet series (refer to \cite{T3} for
related topics). Let $\{\lambda_n\}$
be a sequence in $\Gamma$ such that
$$
0\,\le\,\lambda_1\,<\,\lambda_2\,<\cdots <\lambda_n
\longrightarrow \lambda\,, \qquad n\to\infty\,,
$$
for some $\lambda$ in $\Gamma$. Define a function $\psi$ in $H^2(\sigma)$
by
$$
\psi\;=\;\sum_{n=1}^{\infty}\, a_n \chi_{\lambda_n }
$$
with $\sum_{n=1}^{\infty}\, \vert a_n\vert^2\,<\,\infty$. Observe that
almost every $t \rightarrow \psi(x+e_t)$ extends to an entire function.

\begin{prop}
\label{prop3}
Let $\psi$ be as above and let $\/ \delta >0$. Then there is a decreasing sequence
$\{m_n\}$ with $m_n\rightarrow -\infty$ such that
the number of zeros of $z \rightarrow \psi(x+e_z)$ in the strip
$$
S_n\;=\;\left\{\,z=t+iu\;;\; m_n> u >m_n-\delta \right\}
$$
is infinite, for $\sigma-a.e. \,x$ in $K$.
\end{prop}
\begin{proof}
Putting $\nu_n=\lambda-\lambda_n$, we let
$\phi\;=\;\sum_{n=1}^{\infty}\, \overline{a_n}\, \chi_{\nu_n }$.
Since $z \rightarrow e^{i\lambda z}$ has no zero,
$z \rightarrow \psi(x+e_z)$ has zero at $z$
if and only if so does $z \rightarrow \phi(x+e_z)$ at
$\bar{z}$. For each $r\,>\,0$,
$t \rightarrow \phi \ast P_{ir}(x+e_t)$ cannot be an outer function in $H^2(dt/\pi(1+t^2))$,
even if \/$ \log\vert\phi\,\vert$ does not lie in $L^1(\sigma)$.
Since $\phi$ has no weight at infinity, the inner part of
$t \rightarrow \phi \ast P_{ir}(x+e_t)$
derives a Blaschke cocycle being not constant.
From this fact, we may choose easily a desired decreasing sequence $\{m_n\}$.
\end{proof}

\medskip
[Note on September 9, 2014: Made some changes in wording]

[Note on April 13, 2015: Accepted in ANNALES DE L'INSTITUTE FOURIER.]

\end{document}